\def\version{19 November 2019}

\documentclass[reqno]{amsart}
\usepackage{amssymb,epsfig,graphics,graphicx,bbm,hyperref,color}
\usepackage{pgf,pgfarrows,pgfnodes,pgfautomata,pgfheaps,pgfshade}

\usepackage{multicol}
\usepackage{enumitem}
\usepackage{tikz-cd}
\usetikzlibrary{matrix,arrows,decorations.pathmorphing,positioning}
\usepackage{nicefrac}

\definecolor{gray}{rgb}{0.93,0.93,0.93}
\definecolor{light-gold}{rgb}{0.99,0.97,0.78}

\definecolor{gold}{rgb}{0.7,0.55,0}

\def\be{\begin{equation}}
\def\ee{\end{equation}}
\def\bm{\begin{multline}}
\def\bfig{\begin{figure}[htb]}
\def\efig{\end{figure}}
\newcommand{\dd}{{\rm d}}
\newcommand{\e}[1]{\,{\rm e}^{#1}\,}

\newcommand{\eqd}{\overset{\rm d}{=}}

\numberwithin{equation}{section}
\newtheorem{theorem}{Theorem}[section]
\newtheorem{proposition}[theorem]{Proposition}
\newtheorem{lemma}[theorem]{Lemma}

\newtheorem*{definition}{Definition}
\newtheorem{question}{Question}

\newcommand{\eps}{{\varepsilon}}
\newcommand{\bbC}{{\mathbb C}}
\newcommand{\bbE}{{\mathbb E}}
\newcommand{\bbN}{{\mathbb N}}

\newcommand{\caR}{{\mathcal R}}

\newcommand{\sss}{\scriptscriptstyle}

\makeatletter
  \def\tagform@#1{\maketag@@@{\scriptsize{(#1)}\@@italiccorr}}
\makeatother

\renewcommand{\eqref}[1]{(\ref{#1})}

\newcommand{\EE}{\mathbb{E}}

\renewcommand{\a}{\alpha}

\newcommand{\G}{\Gamma}
\newcommand{\cI}{\mathcal{I}}
\newcommand{\cD}{\mathcal{D}}

\newcommand{\oo}{\infty}
\newcommand{\om}{\omega}

\begin{document}

{\hfill\small \version} \vspace{2mm}

\title{Characterising random partitions by random colouring}

\author{Jakob E. Bj\"ornberg}
\address{Department of Mathematics,
Chalmers University of Technology and the University of Gothenburg,
Sweden}
\email{jakob.bjornberg@gu.se}
 
\author{C\'ecile Mailler}
\address{Department of Mathematical Sciences, University of Bath, Bath BA2 7AY, United Kingdom}
\email{c.mailler@bath.ac.uk}

\author{Peter M\"orters}
\address{Mathematisches Institut, Universit\"at zu K\"oln, Weyertal 86--90, 50931 K\"oln,
Germany}
\email{moerters@math.uni-koeln.de}

\author{Daniel Ueltschi}
\address{Department of Mathematics, University of Warwick,
Coventry, CV4 7AL, United Kingdom}
\email{daniel@ueltschi.org}


%\keywords{}

\begin{abstract}
Let $(X_1,X_2,...)$ be a random
partition of the unit interval $[0,1]$,
i.e.\ $X_i\geq0$ and $\sum_{i\geq1} X_i=1$, 
and let $(\varepsilon_1,\varepsilon_2,...)$ be i.i.d.\ Bernoulli
random variables of parameter $p \in (0,1)$. The {\it Bernoulli
  convolution} of the partition is the random variable 
$Z =\sum_{i\geq1} \varepsilon_i X_i$. 
The question addressed in this article is: Knowing the distribution of
$Z$ for some fixed $p\in(0,1)$, what can we infer about the random partition $(X_1, X_2,...)$?
We consider random partitions formed by 
residual allocation and prove that their distributions are fully characterised
by their Bernoulli convolution if and only if the parameter $p$ is
not equal to $\nicefrac12$.
\end{abstract}

\thanks{\copyright{} 2019 by the authors. This paper may be reproduced, in its
entirety, for non-commercial purposes.}

\maketitle

\section{Introduction}

Random partitions appear in the mathematical description of many
natural systems, such as particle clustering and condensation in
physics \cite{Betz-Ueltschi}; 
dynamics of gene populations in biology \cite{Ewens}; 
wealth distribution
in economics \cite{Sosnovskiy}; etc. 
There is a vast amount of possible probability laws of random
partitions, but one often encounters convergence to one of a few universal 
laws, most  notably the Poisson--Dirichlet distribution with parameter $\theta>0$, 
henceforth denoted~$\mathtt{PD}(\theta)$ and defined 
 below after Eq.\ \eqref{def RA}.

To show convergence of a tight sequence of random partitions it is often feasible to 
show convergence of a derived quantity like the Bernoulli convolutions studied in this paper.
If the limit of the derived quantity characterises the law of the underlying random partition among the class of possible limits, convergence is shown. It is therefore an important question whether the distribution of a random partition can be identified from its Bernoulli convolution, and in this paper we contribute to this problem.

We describe two scenarios that motivate this study in Sections
\ref{sec rand int} and \ref{sec S T}. We introduce the precise setting
and our results in Section \ref{sec setting} --- the definition of the
Bernoulli convolution can be found around Eq.\ \eqref{def Z}. Sections
\ref{sec unique} and \ref{sec non-unique} contain the proofs of our
two theorems. 
We make further comments in Section \ref{sec questions}; it includes a
counterexample due to A.\ Holroyd, that sheds much light on these
questions.

\subsection{Random interchange model and quantum spin systems}
\label{sec rand int}

The random interchange model is a process on permutations
constructed as products of random transpositions. 
Namely, given 
integers $n$ and $k$, we pick $k$ pairs of distinct integers 
$(x_1,y_1),\dotsc,(x_k,y_k)$ from
$\{1,\dotsc,n\}$ uniformly at
random, 
and consider the permutation 
\be
\sigma = \tau_{k} \circ \dots \circ \tau_{1}.  
\ee 
Here,
$\tau_{i}=(x_i,y_i)$ denotes the transposition of $x_i$ and $y_i$.
The cycle structure (i.e.\ the lengths of the permutation
cycles) of $\sigma$ gives an integer partition of 
$n$; dividing by
$n$ gives a partition of $[0,1]$.

Schramm~\cite{Sch} studied this model in the case where 
$k = \lfloor cn\rfloor$ with $c>1$. He proved that, with
high probability as $n\to\infty$, there are cycles whose lengths are
of order $n$. Let $L_i$ denote the length of the $i$th largest
cycle. The sum of cycles of length of order $n$ is $\kappa n (1+o(1))$
with $\kappa = \kappa(c)$ fixed  (and $\kappa\to1$ when $c\to\infty$);
 and the sequence $(\frac{L_1}{\kappa
n}, \frac{L_2}{\kappa n}, \dots)$ converges (weakly) to
$\mathtt{PD}(1)$, the Poisson--Dirichlet distribution with parameter
1. 

One motivation for the random interchange model,
pointed out and exploited by T\'oth \cite{Toth}, is that it provides a
probabilistic representation of the Heisenberg model of quantum
spins. 
For this representation the density of the random interchange model 
gets an extra weight $2^{\#{\rm cycles}}$, which leads to a conjectured limit 
which is the Poisson--Dirichlet distribution $\mathtt{PD}(2)$, see
\cite{GUW}. 
In this case the number of transpositions $k$ 
is random, chosen to be
$\mathtt{Poisson}(cn)$. 
Recently, it was proved in \cite{BFU}
that, in the model with weight $\theta^{\#{\rm cycles}}$,
$\theta=2,3,4,\dots$, we have
\be
\label{result BFU}
\lim_{n\to\infty} \bbE_n \Bigl[ \prod_{i\geq1} \tfrac1\theta (\e{h L_i/n} + \theta - 1) \Bigr] = \e{\frac h\theta (1-\kappa)} \bbE_{\mathtt{PD}(\theta)} \Bigl[ \prod_{i\geq1} \tfrac1\theta (\e{h\kappa X_i} + \theta - 1) \Bigr],
\ee
for some (deterministic) $\kappa \in [0,1]$ 
which depends on $c$ and $\theta$ 
and is positive for $c$ large enough;  
the above identity holds for all $h \in
\bbC$.  
The last
expectation in \eqref{result BFU} is equal to the moment generating
function  at~$h\kappa$ of
the Bernoulli convolution of $\mathtt{PD}(\theta)$ with parameter
$p=\nicefrac1\theta$. 
The interpretation is that the system displays small
(order 1) and large (order~$n$) cycles, and that the joint
distribution of the lengths of large cycles is
$\mathtt{PD}(\theta)$;   see \cite{BFU} for more details. 
But is Eq.~\eqref{result BFU} enough to guarantee 
that the limiting sequence of renormalised cycle lengths 
be equal to $\mathtt{PD}(\theta)$?
We prove here that, among the 
residual allocation
distributions, the answer is yes for $\theta=3,4,\dots$, but no for~$\theta=2$.

There are related loop models that include `double bars' as well as
the transposition `crosses', that represent further quantum spin
systems \cite{AN,Uel}. Without weights, it was proved in \cite{BKLM}
that the joint distribution of the lengths of long loops is
$\mathtt{PD}(\nicefrac12)$. With weights $2^{\#{\rm loops}}$, the
result of \cite{BFU} is that 
\be \lim_{n\to\infty} \bbE_n \Bigl[
\prod_{i\geq1} \cos (h L_i/n) \Bigr] = \bbE_{\mathtt{PD}(1)} \Bigl[
\prod_{i\geq1} \cos(h\kappa X_i) \Bigr], 
\ee 
for all $h \in \bbC$. The
latter expectation is closely related to the moment generating
function of the
Bernoulli convolution of $\mathtt{PD}(1)$ with parameter $p=\nicefrac12$.
Results of the present article show that the above claim is not enough
to guarantee that the limiting distribution is $\mathtt{PD}(1)$, 
even if one assumes that the limiting distribution is a residual allocation.

\subsection{Exchangeable divide-and-color models}
\label{sec S T}
In a recent paper by Steif and Tykesson \cite{ST}, the authors
introduce  \emph{generalized divide-and-color models} as follows.
Given a countable set $S$ and $p\in(0,1)$, one starts by forming a random
partition $\Pi$ of $S$ according to some
rule;  one then assigns to 
each part of $\Pi$ a `color' 0 or 1, independently and
with probability $p$ for 1.
Letting each element of $S$ take the color of the part it belongs
to and then 
forgetting about the original parition $\Pi$, one ends up with a
random element $\om\in\{0,1\}^S$.  This construction is motivated by
the Fortuin--Kasteleyn representation of the Ising model, among other
examples.  

A particular case is when
$S=\bbN$ and when the random partition $\Pi$ 
is \emph{exchangeable}, i.e.\ its distribution is invariant under all
finite permutations of~$\bbN$.  
By Kingman's famous theorem \cite{Kin2}, 
such a random partition of $\bbN$ is uniquely
encoded by a random vector $(X_i)_{i\geq1}$ satisfying 
$X_i\geq X_{i+1}\geq0$ for all $i\geq1$
and \smash{$\sum_{i\geq1} X_i\leq 1$};  note that $<1$ is allowed in this case.
On the other hand, the resulting color process
$\om\in\{0,1\}^\bbN$ is also exchangeable;  by de Finetti's theorem,
this means that there is some random variable $\xi\in[0,1]$ such that,
conditional on $\xi$, the $\om_i$ are i.i.d.\
$\mathtt{Bernoulli(\xi)}$.
It is not hard to see that
(when $\sum_{i\geq1}X_i=1$)
$\xi$ equals the Bernoulli convolution of
$(X_i)_{i\geq1}$, see \cite[Lemma 3.12]{ST}.
Steif and Tykesson ask whether the law of the random partition $\Pi$
can be recovered from the law of $\om$ when $p\neq \nicefrac12$.  
This is equivalent to asking
whether the law of $(X_i)_{i\geq1}$ can be recovered from 
the law of its Bernoulli convolution.
Our results on residual
allocation models show that the answer can be \emph{yes} under
additional assumptions on $(X_i)_{i\geq1}$.

\subsection{Framework and results}
\label{sec setting}

We define a Bernoulli convolution as follows.

\begin{definition}
Let $(X_i)_{i\geq 1}$ be a random partition of $[0,1]$,
i.e.\ $X_i\geq 0$ for all $i\geq 1$ and $\sum_{i\geq 1}X_i = 1$.
Let $(\eps_i)_{i\geq1}$  be a sequence of i.i.d.\ Bernoulli
random variables of parameter $p \in (0,1)$, independent of 
$(X_i)_{i\geq1}$. Set
\be
\label{def Z}
Z = \sum_{i\geq1} \eps_i X_i.
\ee 
The law of $Z$, and sometimes the random variable $Z$ itself, is
called the {\it Bernoulli($p$) convolution} of the random 
partition $(X_i)_{i\geq1}$. 
\end{definition}

We restrict our setting to random partitions obtained from {\it
residual allocation}. Namely, we consider the interval $[0,1]$ with
the Borel $\sigma$-algebra. Given a probability measure $\mu$ on
$[0,1]$, let
$(Y_i)_{i\geq1}$ be i.i.d.\ random variables distributed according to
$\mu$, and consider the sequence $(X_i)_{i\geq1}$ defined by
\be
\label{def RA}
\begin{split}
&X_1 = Y_1, \\
&X_2 = (1-Y_1) Y_2, \\
&X_3 = (1-Y_1) (1-Y_2) Y_3, \\
&\text{etc...}
\end{split}
\ee
Assuming that $\mu(\{0\}) < 1$, it is not hard to prove that 
$X_i \to 0$ as $i \to \infty$ and that $\sum_{i\geq1} X_i = 1$, almost
surely. It is possible to rearrange the sequence $(X_i)_{i\geq 1}$ in decreasing
order if one wants an ordered partition, but this is not necessary
here.

An important example of this construction is the Griffiths, 
Engen and McCloskey distribution,
$\mathtt{GEM}(\theta)$, obtained when $\mu=\mathtt{Beta}(1,\theta)$.
If one orders the entries of a $\mathtt{GEM}(\theta)$ sample by
decreasing size, one obtains the famous Poisson--Dirichlet
distribution $\mathtt{PD}(\theta)$, see \cite{Kin}.  
Another important example is the `classical' Bernoulli 
convolution $\sum_{i\geq1} \pm \lambda^i$ with i.i.d.\ random signs; see the review~\cite{PSS}.
This falls into our framework 
(take $\mu=\delta_{1-\lambda}$ for some fixed $\lambda\in(0,1)$ so that
$X_i=(1-\lambda)\lambda^{i-1}$), 
except that our Bernoulli coefficients take value in $\{0,1\}$ instead
of $\{-1, 1\}$.

As a shorthand, since we only consider random partitions from
residual allocation, we will sometimes refer to $Z$ (or its law) as the
Bernoulli convolution of the measure $\mu$.
The Bernoulli convolution 
is invariant under rearrangements of
the sequence $(X_i)_{i\geq 1}$. The cases $p=0$ and $p=1$ are trivial and
uninteresting, since $Z=0$ and $Z=1$, respectively.

If $\mu$ has an atom at 0 of value $c>0$, i.e.\ $\mu(\{0\}) = c$, then
the sequence $(Y_1,Y_2,\dots)$ --- and therefore $(X_1,X_2,\dots)$ ---
contains a density $c$ of elements that are equal to 0; this does not
affect $Z$. In other words, the Bernoulli convolutions of 
$\mu$ and 
$c \delta_0 + (1-c) \mu$ are the same for all $c \in [0,1)$. We avoid this
trivial degeneracy by restricting our attention to measures that do
not have an atom at 0.

Given $p \in (0,1)$, the question is whether the Bernoulli($p$) convolution
characterises the random partition obtained from
residual allocation. We show that it is the case for 
$p \neq \nicefrac12$.

\begin{theorem}
\label{thm unique}
Let $p \in (0,1) \setminus \{ \nicefrac12 \}$. If $\mu$ and $\nu$ are two
probability measures on $[0,1]$ such that 
$\mu(\{0\}) = \nu(\{0\}) =0$, 
and the corresponding residual allocation models 
have identical Bernoulli($p$) convolution,
then $\mu = \nu$.
\end{theorem}

We also show that Theorem \ref{thm unique} fails for
$p=\nicefrac12$. Our non-uniqueness results hold for $\mathtt{GEM}$
(or Poisson--Dirichlet) measures of arbitrary parameters.

\begin{theorem}
\label{thm non-unique} 
Let $\theta>0$ and $\mu = \mathtt{Beta}(1,\theta)$. 
Then there exist
infinitely many
 $\nu \neq \mu$ such that
$\nu(\{0\}) = 0$, and such that 
$\mu$ and $\nu$ have identical
Bernoulli($\nicefrac12$) convolutions.
\end{theorem}

The non-uniqueness results are not explicit 
with the exception of $\mathtt{GEM}(2)$: 
We show that if an (absolutely continuous) measure
$\nu$ satisfies
\be
\label{das ist genug}
x \,\dd\nu(x) = (1-x) \,\dd\nu(1-x) \quad
\text{ for all }x \in [0,1],
\ee
then its residual allocation has the same Bernoulli convolution
as $\mu=\mathtt{Beta}(1,2)$.
Note that  \eqref{das ist genug} holds true in the case 
$\mu=\mathtt{Beta}(1,2)$, for which
$\dd\mu(x) = 2(1-x) \dd x$.
Another example is the Dirac measure at $x=\nicefrac12$, $\nu = \delta_{\nicefrac12}$, which formally satisfies \eqref{das ist genug}.
We refer to Proposition \ref{prop GEM(2)} for
details including conditions on the regularity of measures.

We prove Theorems \ref{thm unique} and \ref{thm non-unique} with the
help of a stochastic identity for the random variable~$Z$, see Lemma
\ref{lem stoch id}. This identity holds because of the self-similarity
structure of
residual allocations. The proofs of Theorem
\ref{thm unique} and \ref{thm non-unique} can be found in Sections
\ref{sec unique} and~\ref{sec non-unique}, respectively.

A natural question is whether Theorem \ref{thm unique} holds beyond
residual allocations.  Obviously, the Bernoulli convolution
\eqref{def Z} may be defined for arbitrary random partitions
$(X_i)_{i\geq1}$.  
Alexander Holroyd has given an example showing that, in general,
the Bernoulli convolution does \emph{not} determine the random
partition, even if the former is known for all $p\in(0,1)$;
we explain Holroyd's example in Section \ref{sec questions}.
One may also allow more general random variables
$(\eps_i)_{i\geq1}$;  in this generality, $Z$ is sometimes called a
\emph{random weighted average}.  Pitman's recent review \cite{Pit3}
contains a wealth of information about the theory of random weighted
averages.
In \cite[Corollary 9]{Pit3} it is shown that the
distributions of the random weighted averages $Z$, as
$(\eps_i)_{i\geq1}$ range over all i.i.d.\ sequences of random
variables with finite support, fully characterize the law of the
random partition $(X_i)_{i\geq1}$.  This holds without any assumptions
about the properties of the random partition.  
It is natural to ask whether the condition on the $\eps_i$ can
  be weakened.

\section{Uniqueness when $p\neq\nicefrac12$ (proof of Theorem~\ref{thm unique})}
\label{sec unique}

The following lemma will be used both to establish uniqueness for
$p\neq\nicefrac12$ and non-uniqueness for $p=\nicefrac12$.

\begin{lemma}\label{lem stoch id} Let $Y, Y_1, Y_2, \dots$ be i.i.d.\
random variables with values in $[0,1]$ and $(X_i)_{i\geq1}$ defined by
\eqref{def RA}; $\eps,\eps_1, \eps_2, \dots$ be i.i.d.\
$\mathtt{Bernoulli}(p)$ random variables independent of the $Y$'s; 
and $Z$ and $Z'$ be two identically distributed random variables
with values in $[0,1]$, $Z'$ being independent of $Y$ and $\eps$. 
The following stochastic identities are
equivalent:
\begin{itemize}
\item[(a)] $\displaystyle Z \eqd \sum_{i\geq1} \eps_i X_i$;
\item[(b)] $\displaystyle Z \eqd \eps Y + (1-Y) Z'$.
\end{itemize}
\end{lemma}

This is not new, see \cite[Theorem 1]{FT} or \cite[Theorem 7.1]{DF}; it is also discussed in \cite[(119)]{Pit3}.

\begin{proof}
Assuming $(a)$, we have
\[
Z \eqd \eps_1 Y_1 + (1-Y_1) \sum_{i\geq2} \eps_i \frac{X_i}{1-Y_1},
\]
where the sequence $(X_i/(1-Y_1))_{i\geq 2}$ is independent of 
$X_1= Y_1$ and has the same distribution as $(X_i)_{i\geq 1}$, 
which gives (b).

Assuming $(b)$, we construct a sequence of random
variables which all have the same distribution as $Z$ and which
converge weakly (in fact, almost surely) 
to $\sum_{i\geq1} \eps_i X_i$.  Observe that there exist
$Z_1$ and $Z_2$ two independent copies of $Z$, 
independent of $\eps_i$ and $Y_i$ such that
\be
\begin{split}
Z_1 &\eqd \eps_1 Y_1 + (1-Y_1) Z_1 \\
&\eqd \eps_1 Y_1 + (1-Y_1) \bigl[ \eps_2 Y_2 + (1-Y_2) Z_2 \bigr].
\end{split}
\ee
Iterating this further, we get $(Z_i)_{i\ge 1}$ such that for all $n\geq1$, 
\be
\sum_{i=1}^n \eps_i X_i + (1-Y_1) \dotsb (1-Y_n) Z_n \eqd Z_1,
\ee
where  $X_i$ are as defined in \eqref{def RA}.
All terms in $\sum_{i=1}^n \eps_i X_i$ are positive and the sums are
bounded by 1, hence the series converges to
$\sum_{i\geq1} \eps_i X_i$;   the remainder 
$ (1-Y_1) \dotsb (1-Y_n) Z_n$ converges to 0 almost surely.
As $n\to\infty$ we obtain (a).
\end{proof}

We will show that all moments of $Y \sim \mu$ are determined by the
Bernoulli convolution $Z$ of the residual allocation model from
$\mu$. This holds for $p \in (0,1) \setminus \{\nicefrac12\}$. It does
not hold for $p=0$ (the Bernoulli convolution is always 0) and $p=1$
(it is always 1). It also does not hold for $p=\nicefrac12$, for
reasons that are not obvious and that are discussed 
 in Sect.\ \ref{sec non-unique}.

Let us introduce numbers $a_{n,k}$ and $c_n$ that depend on the law of $Z$, 
and numbers $b_n$ that depend on the law of~$Y$. 
For $n,k \in \bbN$ with $k \leq n$, let
\be
\begin{split}
&a_{n,k} = (-1)^k p \binom nk \bbE \bigl[ (1-Z)^k \bigr], \\
&c_n = (1-p) \bbE[Z^n], \\
&b_n = \frac{1 - \bbE \bigl[ (1-Y)^n \bigr]}{\bbE[Y]}.
\end{split}
\ee
Note that $b_0=0$, $b_1=1$, and $a_{1,1}+c_1=0$ 
since $\bbE[Z]=p$. We have the following relations.

\begin{proposition}
\label{prop Y and Z}
For all $p \in [0,1]$ and all $n \geq 1$, we have
\[
c_n b_n + \sum_{k=1}^n a_{n,k} b_k = 0.
\]
\end{proposition}

\begin{proof}
We expand $\bbE[Z^n]$ in two different ways. First,
\be\begin{split}
\bbE[Z^n] &= 
(1-p) \bbE[Z^n]+
p \bbE \bigl[ \bigl( 1 - (1-Z) \bigr)^n \bigr] \\
&= (1-p) \bbE[Z^n]+
p \sum_{k=0}^n (-1)^k \binom nk \bbE \bigl[ (1-Z)^k \bigr].
\end{split}\ee
Second, using Lemma \ref{lem stoch id},
\be\begin{split}
\EE [Z^n]
&=\EE \bigl[ \big(\eps Y+(1-Y)Z\big)^n \bigr]
=p \EE \bigl[ \big(Y+(1-Y)Z\big)^n \bigr]
+(1-p) \EE \bigl[ \big((1-Y)Z\big)^n \bigr] \\
&= (1-p) \EE \bigl[ (1-Y)^n \bigr] \EE [Z^n] +
p \EE \bigl[ \big(1-(1-Y)(1-Z)\big)^n \bigr] \\
&=(1-p) \EE \bigl[ (1-Y)^n \bigr] \EE [Z^n] +
p\sum_{k=0}^n (-1)^k \binom{n}{k} \EE \bigl[ (1-Y)^k \bigr] \EE \bigl[ (1-Z)^k \bigr].
\end{split}\ee
Equating these identities, we get
\be
0 = (1-p) \bbE[Z^n] \bigl\{ 1 - \bbE \bigl[ (1-Y)^n \bigr] \bigr\} + p \sum_{k=0}^n (-1)^k \binom nk \bbE \bigl[ (1-Z)^k \bigr]  \bigl\{ 1 - \bbE \bigl[ (1-Y)^k \bigr] \bigr\}.
\ee
We now divide by $\bbE[Y]$ and we obtain the claim of the proposition.
\end{proof}

The next lemma holds for $p \neq \nicefrac12$ only.

\begin{lemma}\label{lem P-coeff}
For $p \in (0,1) \setminus \{\nicefrac12\}$, we have for all $n\geq2$ that
\[
a_{n,n} + c_n \neq 0.
\] 
\end{lemma}

\begin{proof}
We have
\be
a_{n,n} + c_n = (1-p)\EE [Z^n] +(-1)^np\EE \bigl[ (1-Z)^n \bigr].
\ee
This is always positive for $n$ even; we thus assume from now on that~$n\geq3$ is odd.  From the definitions \eqref{def RA} and \eqref{def Z}, we have
\be
\mathbb E [Z^n]=\sum_{i_1,i_2,\dotsc,i_n\geq 1}
\mathbb E\big[\eps_{i_1}\eps_{i_2}\dotsb \eps_{i_n}\big]
\mathbb E\big[X_{i_1}X_{i_2}\dotsb X_{i_n}\big].
\ee
Note that, if $\ell=\#\{i_1,\dotsc,i_n\}$ denotes
the number distinct indices among $i_1,\dotsc, i_n\geq 1$, then
\be
\mathbb E\big[\eps_{i_1}\eps_{i_2}\dotsb \eps_{i_n}\big]
= p^\ell,
\ee
since $\eps_i^k=\eps_i$ for all $k, i\geq1$.
We thus get
\be
\EE [Z^n]=\sum_{\ell=1}^n p^\ell \EE [S_{n,\ell}],
\ee
where
$S_{n,\ell}=\sum X_{i_1}X_{i_2}\dotsb X_{i_n}$
summed over all choices of indices $i_1,\dotsc, i_n\geq 1$ such that 
$\#\{i_1,\dotsc,i_n\}=\ell$.
Note that $\EE [S_{n,\ell}]>0$ for all $\ell\geq 1$.
Since, by definition, $1-Z=\sum_{i\geq1} (1-\eps_i)X_i$ we also have
$\EE [(1-Z)^n] = \sum_{\ell=1}^n (1-p)^\ell \EE [S_{n,\ell}]$, and thus
\be
a_{n,n} + c_n = p(1-p) \, \EE [S_{n,\ell}] \, \sum_{\ell= 1}^n \bigl( p^{\ell-1}-(1-p)^{\ell-1} \bigr).
\ee
While the term $\ell=1$ is zero, all other terms are non-zero and
have the same sign, which proves the claim since $n>1$.
\end{proof}

We now turn to the proof of Theorem \ref{thm unique}.

\begin{proof}[Proof of Theorem \ref{thm unique}]
It follows from Proposition \ref{prop Y and Z} and Lemma 
\ref{lem P-coeff} that, for $n\geq2$,
\be\label{eq b-rec}
b_n = -(a_{n,n}+c_n)^{-1} \sum_{k=1}^{n-1} a_{n,k} b_k.
\ee
Recall that $b_0=0$, $b_1=1$. The above equation shows that the $b_n$'s
are recursively determined by the $a_{n,k}$'s and $c_n$'s, which only depend on the
Bernoulli convolution $Z$. As $n\to\infty$, the sequence $(b_n)$
converges to $1/\bbE[Y]$ --- here we use our assumption that the
measure $\mu$ does not have an atom at $0$. It follows that $\bbE[Y]$
and $\bbE[(1-Y)^n]$ are determined by the Bernoulli convolution for
all $n$. Then all moments of the original measure $\mu$ are known,
hence the measure $\mu$ itself (see \cite
[Theorem~1.2]{billingsley}). 
\end{proof}

\section{Non-uniqueness when $p=\nicefrac12$ (proof of Theorem~\ref{thm non-unique})}  
\label{sec non-unique}

In this section we set $p=\nicefrac12$, unless indicated otherwise.  
We also assume that the Bernoulli convolution of parameter $\nicefrac12$ 
has a density $q(x)$ 
with respect to  Lebesgue measure,
and that $q(x)>0$ for all $x \in (0,1)$. 
This will hold in particular in the case of
$\mathtt{GEM}(\theta)$. 
Since $p=\nicefrac12$ we then have that 
$q(x) = q(1-x)$ because \smash{$Z \eqd 1-Z$}.

Given a nonnegative measurable function $\rho$ on $[0,1]$, we define
the function $H\rho$ by 
\be
\label{def op H}
[H\rho](x) = \frac1{q(x)} \int_0^x q\Bigl( \frac{x-u}{1-u} \Bigr) \frac{\rho(u)}{1-u} \dd u.
\ee
Let $\caR_q$ be the cone of nonnegative measurable functions $\rho$
such that the integral above is finite 
for all $0\leq x\leq 1$.
$H$ is a linear operator on $\caR_q$. As it turns out, it
gives a relation between the density $\rho$ of a probability measure
on $[0,1]$, and the density $q$ of the corresponding Bernoulli
convolution.  This may be seen by expanding the stochastic
  identity of Lemma \ref{lem stoch id} (b) and making a suitable
  change of variables.  More precisely, we have:

\begin{lemma}
\label{lem characterisation}
Let $q$ be a probability density function on $[0,1]$ such that $q(x) >
0$ on $(0,1)$ and $q(x) = q(1-x)$. Let $\rho \in \caR_q$;  we have
\be\label{eq H}
[H\rho](x)+[H\rho](1-x)=2, \quad
\text{for almost all } x \in [0,1],
\ee
if and only if
\begin{itemize}
\item[(a)] $\rho$ is a probability density function on $[0,1]$, and
\item[(b)] the Bernoulli($\nicefrac12$) convolution of the residual
  allocation model from $\rho$ has density $q$. 
\end{itemize}
\end{lemma}

\begin{proof}
Assume that \eqref{eq H} holds.
For (a), we have, writing $h(x)=[H\, \rho](x)$,
\be\begin{split}
1&=
\int_0^1 q(x) \tfrac{h(x)+h(1-x)}{2}\,\dd x 
=\int_0^1 q(x) h(x) \,\dd x
=\int_0^1 \dd u\,\rho(u)\int_u^1 \dd z
\tfrac1{1-u}q\big(\tfrac{z-u}{1-u}\big)\\
&=\int_0^1 \dd u\,\rho(u)\int_0^1 \dd v\, q(v) 
=\int_0^1 \dd u\,\rho(u),
\end{split}\ee
as claimed.  (We used the change of variables $v=\tfrac{z-u}{1-u}$.)

For (b), we use \eqref{eq H} to get
\be\label{eq:tmp2}
q(x) = \tfrac12 \int_0^x q\Bigl( \frac{x-u}{1-u} \Bigr) \frac{\rho(u)}{1-u} \dd u + \tfrac12 \int_0^{1-x} q\Bigl( \frac{x}{1-u} \Bigr) \frac{\rho(u)}{1-u} \dd u.
\ee
It follows that for all continuous function $f$, we have
\be\label{eq:tmp1}
\begin{split}
\int_0^1 & q(x) f(x) \, \dd x\\
& = \tfrac12 \int_0^1 \frac{\rho(u)}{1-u} \dd u \int_u^1 q \Bigl( \frac{x-u}{1-u} \Bigr) f(x) \dd x + \tfrac12 \int_0^1 \frac{\rho(u)}{1-u} \dd u \int_0^{1-u} q \Bigl( \frac{x}{1-u} \Bigr) f(x) \dd x \\
&=  \tfrac12 \int_0^1 \rho(u) \dd u \int_0^1 f \bigl( u + (1-u)y \bigr) q(y) \dd y + \tfrac12 \int_0^1 \rho(u) \dd u \int_0^1 f \bigl( (1-u)y \bigr) q(y) \dd y.
\end{split}
\ee
We used Fubini's theorem to get the second line, and the changes of
variables $y = \frac{x-u}{1-u}$ and $y = \frac{x}{1-u}$ (for fixed
$u$) to get the third line. The left side gives the expectation
$\bbE[f(Z)]$ for the random variable with density $q$. The right side
gives $\bbE[f(\eps Y+(1-Y)Z)]$ for the independent random variables
$\eps \sim \mathtt{Bernoulli}(\frac12)$, $Y$ with density $\rho$, and
$Z$ with density $q$. We recognise the stochastic identity of Lemma
\ref{lem stoch id} (b). Hence $q$ is the density of the Bernoulli
convolution of~$\rho$. 

The other implication can be checked similarly:
\eqref{eq:tmp1} holds by (b), hence also \eqref{eq:tmp2} for
  almost all $x$, which gives \eqref{eq H}.
\end{proof}

The next step is to identify the Bernoulli convolution of
$\mathtt{GEM}$ distributions. It turns out to be equal to Beta random
variables. We consider general parameters $p$, although we only need
the case $p = \nicefrac12$ here.

\begin{proposition}\label{prop gem-beta}
Let $\theta>0$ and $p\in[0,1]$. Then the Bernoulli convolution of $\mathtt{GEM}(\theta)$, i.e.\ of the residual allocation model from $\mathtt{Beta}(1,\theta)$ random variables, is the $\mathtt{Beta}(p\theta,(1-p)\theta)$ distribution.
\end{proposition}

This result is not new, see e.g.\ \cite[Prop.\ 27(iii)]{Pit3}.
We sketch a proof using
the connection between $\mathtt{GEM}(\theta)$
and $\mathtt{PD}(\theta)$, 
Kingman's characterization of $\mathtt{PD}(\theta)$ in
terms of the Gamma-subordinator, as well as  the following
well-known lemma (see e.g.\ \cite[Lemma 7.4]{GUW}):

\begin{lemma}\label{beta-gamma-lem}
If $Y_1$ and $Y_2$ are independent, with respective distributions
$\mathtt{Gamma}(\theta_1,1)$ and $\mathtt{Gamma}(\theta_2, 1)$, then
\begin{enumerate}
\item $Y_1+Y_2$ has distribution $\mathtt{Gamma}(\theta_1+\theta_2,1)$,
\item $Y_1/(Y_1+Y_2)$ has distribution $\mathtt{Beta}(\theta_1,\theta_2)$,
\item $Y_1+Y_2$ and $Y_1/(Y_1+Y_2)$ are independent.
\end{enumerate}
\end{lemma}

\begin{proof}[Proof of Proposition \ref{prop gem-beta}]
Let $\xi=(\xi_1,\xi_2,\dotsc)$ be the points of a Poisson process with
intensity measure $\theta x^{-1}\mathrm{e}^{-x}\, \dd x$ on $(0,\oo)$ 
in decreasing order. 
Let $S = \sum_{i\geq1} \xi_i$ and $X_i = \xi_i/S$ for all $i\geq 1$,
then  $S\sim\mathtt{Gamma}(\theta,1)$ and $X = (X_1, X_2, \ldots)$ is
$\mathtt{PD}(\theta)$-distributed  (see~\cite[Definition~2.5]{Ber}).
Let $(\varepsilon_i)_{i\geq 1}$ be a sequence of i.i.d.\ $\mathtt{Bernoulli}(p)$ random variables. Let $\xi^{\sss (1)}$ be the collection $(\xi_i\colon \varepsilon_i = 1)$ and $\xi^{\sss (0)}$ its complement $(\xi_i \colon \varepsilon_i = 0)$.
Note that  $\xi^{\sss (1)}$ and $\xi^{\sss (0)}$ are independent Poisson processes with respective intensity measures $p\theta x^{-1}\mathrm e^{-x}\, \dd x$ and 
$(1-p)\theta x^{-1}\mathrm e^{-x}\, \dd x$ on $(0,\oo)$.
Set \smash{$Y_1=\sum_{i\geq1} \xi^{\sss (1)}_i$} and 
\smash{$Y_0=\sum_{i\geq1} \xi^{\sss (0)}_i$}.  Then $Y_1$ and $Y_0$ have
distributions $\mathtt{Gamma}(p\theta,1)$ and $\mathtt{Gamma}((1-p)\theta,1)$
respectively (this can be checked using the Laplace transform and
Campbell's formula as in~\cite[Lemma 7.3]{GUW}).
Since $Z = Y_1/(Y_0+Y_1)$, Lemma \ref{beta-gamma-lem} implies that 
$Z\sim \mathtt{Beta}(p\theta, (1-p)\theta)$, which concludes the
proof. 
\end{proof}

We now consider a special case of Theorem~\ref{thm non-unique}, namely $\theta=2$.

\begin{proposition}
\label{prop GEM(2)}
Let $\rho$ be a probability density function on $[0,1]$ such that
$\int_0^1 \frac{\rho(u)}{1-u} \, \dd u < \infty$. Then the corresponding
residual allocation model has the same Bernoulli($\nicefrac12$)
convolution as $\mathtt{GEM}(2)$, if and only if 
\be
\label{cond GEM(2)}
x \rho(x) = (1-x) \rho(1-x)\quad\text{for almost all } x \in [0,1].
\ee

\end{proposition}

Note that
there exist many solutions to \eqref{cond GEM(2)}: Starting from an
arbitrary nonnegative integrable function $f$ on $[0,\frac12]$, one can set $f(x)
= \frac{1-x}x f(1-x)$ for $x \in (\frac12,1]$ and take $\rho(x) = f(x)
/ \int f$.  As mentioned before, 
the density of the $\mathtt{Beta}(1,2)$ random
variable is $2(1-x)$ and it satisfies Eq.\ \eqref{cond GEM(2)}.

\begin{proof}
The  Bernoulli$(\nicefrac12)$ convolution of $\mathtt{GEM}(2)$ is equal to $\mathtt{Beta}(1,1)$, i.e.\ the uniform probability measure on $[0,1]$, by Proposition \ref{prop gem-beta}. The operator $H$ takes a simpler form and Eq.\ \eqref{eq H} becomes
\be
\int_0^x \frac{\rho(u)}{1-u}\,\dd u + \int_0^{1-x} \frac{\rho(u)}{1-u}\,\dd u =2,
\ee
for all $x\in[0,1]$. We get \eqref{cond GEM(2)} by differentiating with respect to $x$. This proves the ``only if" direction.

Conversely, if $\rho$ is a probability density function on $[0,1]$ that satisfies Eq.\ \eqref{cond GEM(2)}, then
\be
\begin{split}
[H\rho](x) + [H\rho](1-x) &= \int_0^x \frac{\rho(u)}{1-u} \dd u + \int_0^{1-x} \frac{\rho(u)}{1-u} \dd u = \int_0^1 \frac{\rho(u)}{1-u} \dd u \\
&= \int_0^1 \frac{\rho(u)}{1-u} (1-u+u) \dd u = \int_0^1 \rho(u) \dd u + \int_0^1 \rho(1-u) \dd u = 2,
\end{split}
\ee
and \eqref{eq H} holds true.
\end{proof}

The case of the $\mathtt{GEM}(\theta)$ distribution with $\theta\neq2$ is more complicated and we do not give a full characterisation of all possibilities. We only prove the existence of many solutions.

We rely on the theory of fractional derivatives and integrals, see
e.g.\ \cite[Ch 1]{SKM} for an extended exposition. For $\alpha>0$, let
$\cI^\alpha$ denote the \emph{fractional integral operator} (in the
sense of 
Riemann--Liouville): 
\be
[\cI^\a f](x)=\frac1{\G(\a)}\int_0^x \frac{f(u)}{(x-u)^{1-\alpha}} \, \dd
u, \ee for all $x\in[0,1]$ and all functions $f$ such that the above
integral converges absolutely. Its inverse is the \emph{fractional
derivative operator} $\cD^\a$.  Writing
$\alpha = [\a] + \{\a\}$ with $[\a] \in \bbN_0$ and $\{\a\} \in [0,1)$, 
 it is given by 
\be
[\cD^\a f](x)=\frac1{\G(1-\{\a\})} \frac{\dd^{[\a]+1}}{\dd x^{[\a]+1}}
\int_0^x\frac{f(t)}{(x-t)^{\{\a\}}} \, \dd t.
\ee 
We introduce the function $\varphi$ on $[0,1]$ by 
\be 
\varphi(u) =
\frac{\rho(u)}{(1-u)^{\theta-1}}.  
\ee 
We now rewrite Eq.\ \eqref{eq H} using the fractional 
integral operator in the case where the
probability density $q$ is that of
$\mathtt{Beta}(\nicefrac\theta2,\nicefrac\theta2)$. Taking $q(x) =
\frac{\Gamma(\theta)}{\Gamma(\nicefrac\theta2)^2}
x^{\nicefrac\theta2-1} (1-x)^{\nicefrac\theta2-1}$ in Eq.\ 
\eqref{def op H}, Lemma~\ref{lem characterisation} can be reformulated as
follows.

\begin{lemma}
\label{lem partial int condition}
Let $\theta>0$. Assume that $\varphi$ is a nonnegative function on $[0,1]$ that satisfies
\be
\label{skvoll!}
\frac1{x^{\nicefrac\theta2-1}} [\cI^{\nicefrac\theta2} \varphi](x) +
\frac1{(1-x)^{\nicefrac\theta2-1}} [\cI^{\nicefrac\theta2}
\varphi](1-x) = \frac2{\Gamma(\nicefrac\theta2)} \quad
\text{for all } x\in[0,1].
\ee
Then $\rho(x) = (1-x)^{\theta-1} \varphi(x)$ is a probability function
on $[0,1]$ and the Bernoulli($\nicefrac12$) convolution 
of the residual allocation model from $\rho$ has
density $\mathtt{Beta}(\nicefrac\theta2,\nicefrac\theta2)$. 
\end{lemma}

The claim about non-uniqueness, Theorem \ref{thm non-unique}, is now a consequence of Lemma \ref{lem partial int condition}.

\begin{proof}[Proof of Theorem \ref{thm non-unique}]
We are looking for nonnegative solutions $\varphi$ of \eqref{skvoll!};
then $\rho(x) = (1-x)^{\theta-1} \varphi(x)$ is a solution. Let $\eps$
be a function on $[0,1]$ that is antisymmetric around $\nicefrac12$, i.e.\
$\eps(x) = -\eps(1-x)$, and consider the equation
\be
\label{eq for phi}
[\cI^{\nicefrac\theta2} \varphi](x) = \frac2{\Gamma(\nicefrac\theta2)} \bigl[ x^{\nicefrac\theta2} + x^{\nicefrac\theta2-1} \eps(x) \bigr]
\ee
with $x \in [0,1]$. Solutions of this equation are also solutions of
\eqref{skvoll!}. Applying the fractional derivative operator on both
sides, and using 
$\cD^\alpha \cI^\alpha = \mathtt{id}$, we get
\be
\begin{split}
\varphi(x) &= \frac2{\Gamma(\nicefrac\theta2)} \cD^{\nicefrac\theta2} \bigl[ x^{\nicefrac\theta2} + x^{\nicefrac\theta2-1} \eps(x) \bigr](x) \\
&= \frac2{\Gamma(\nicefrac\theta2) \Gamma(1-\{\nicefrac\theta2\})} \frac{\dd^{[\nicefrac\theta2]+1}}{\dd x^{[\nicefrac\theta2]+1}} \int_0^x \frac{t^{\nicefrac\theta2} + t^{\nicefrac\theta2-1} \eps(t)}{(x-t)^{\{\nicefrac\theta2\}}} \dd t.
\end{split}
\ee
Conversely, if we assume in addition that $\eps(x) = \mathcal{O}(x)$ at $x=0$, we can use \cite[Eq.\ (2.60)]{SKM} to verify that \eqref{eq for phi} is satisfied. Indeed, all derivatives in \cite[Eq.\ (2.60)]{SKM} vanish at $x=0$.

The contribution of the term $t^{\nicefrac\theta2}$ can be calculated explicitly; it gives the constant $\theta$. We can also make the change of variables $t \mapsto ux$ and we get
\be
\varphi(x) = \theta + \frac2{\Gamma(\nicefrac\theta2) \Gamma(1-\{\nicefrac\theta2\})} \frac{\dd^{[\nicefrac\theta2]+1}}{\dd x^{[\nicefrac\theta2]+1}} \biggl[ x^{[\nicefrac\theta2]} \int_0^1 \frac{u^{\nicefrac\theta2-1} \eps(ux)}{(1-u)^{\{\nicefrac\theta2\}}} \dd u \biggr].
\ee
The case $\eps\equiv0$ leads to $\varphi(x) = \theta$, i.e.\ $\rho = \mathtt{Beta}(1,\theta)$.
But we can also choose $\eps\not\equiv0$ to be small and smooth enough such that the last
term is uniformly bounded by $\theta$. Then $\varphi(x) \geq0$ for all
$x\in[0,1]$. 
\end{proof}

\section{Comments}
\label{sec questions}

\subsection{Other examples of non-uniqueness for $p=\nicefrac12$}

For $p=\nicefrac12$ there is another example of non-uniqueness of
the Bernoulli convolution for $\mathtt{GEM}(2)$, using the Brownian
bridge.  Namely, let $X_1\geq X_2\geq\dotsb$ be a ranked list of the
excursion lengths away from 0 of a standard Brownian bridge on
$[0,1]$, and let $\eps_i$ be the indicator that the bridge is positive
on the corresponding excursion.  Then the 
$\eps_i$ are i.i.d.\ 
$\mathtt{Bernoulli}(\nicefrac12)$, independent of the
$X_i$, and the Bernoulli($\nicefrac12$) 
convolution $Z=\sum_{i\geq1}\eps_i X_i$ equals
the time spent positive by the bridge.  L\'evy showed that the latter
is uniformly distributed on $[0,1]$, which as we saw 
coincides with the
Bernoulli($\nicefrac12$) convolution of $\mathtt{GEM}(2)$.
See e.g.\ \cite[Section~2.4]{Pit3} for more information.

We can also use the Brownian pseudo-bridge to get an example of
non-uniqueness of the Bernoulli($\nicefrac12$) convolution for
$\mathtt{GEM}(1)$. Indeed, the ranked list of excursions is given by
the two-parameter Poisson-Dirichlet distribution
$\mathtt{PD}(\nicefrac12,0)$ and the time spent positive is
$\mathtt{Beta}(\nicefrac12,\nicefrac12)$; see \cite{PY}.

\subsection{Further questions}

It would be interesting to investigate the extent to which Theorems
\ref{thm unique} and \ref{thm non-unique} hold for other classes of
random partitions $(X_i)_{i\geq1}$ than those formed by 
residual allocation.  
One could for example consider more general residual allocation
models where the sequence $(Y_i)$ is not i.i.d.\ but e.g.\
given by a discrete-time stochastic  process.
Another natural class of random partitions 
are those built from
subordinators (see \cite[Section 5.2]{Pit3}).
Briefly, in this case $(X_i)_{i\geq1}$ is formed by normalising 
an exhaustive list of the jumps of a subordinator with no drift
component.  We pose the following two questions:

\begin{question}
Are there analogs of Theorems \ref{thm unique} and 
\ref{thm non-unique}  for random partitions built from subordinators? 
\end{question}

\begin{question}
For $p \neq\nicefrac12$,
are there natural examples of 
random partitions whose Bernoulli convolutions
are identical to those of $\mathtt{GEM}(\theta)$, or other 
residual allocation models? 
\end{question}

\subsection{Holroyd's example}

If one makes \emph{no} assumptions about the structure of the
partition $(X_i)_{i\geq1}$ then the Bernoulli convolution does not
determine the law of the random partition, even if the former is known
for \emph{all} $p\in(0,1)$.  This is shown by the following example
due to A.\ Holroyd.

The example deals with partitions of just three elements.  We
consider random variables $X_1,X_2,X_3$ such that $X_1 \geq X_2 \geq
X_3 \geq 0$ and $X_1 + X_2 + X_3 =1$, as well as independent
\texttt{Bernoulli}($p$) random variables $\eps_1,\eps_2,\eps_3$.  
The first observation is that
the law of the Bernoulli($p$)-convolution 
$Z$ is determined by the marginals for $X_1$, $X_2$, and
$X_1+X_2$ --- no matter what the values of $\eps_1,\eps_2,\eps_3$ are,
the random variables $X_1$ and $X_2$ appear in the above form.  This
holds for all $p$.  It is thus enough to show that we can find random
variables $\tilde X_1$ and $\tilde X_2$, distinct from $X_1,X_2$, such
that
\be
\begin{split}
&\tilde X_1 \eqd X_1 \\
&\tilde X_2 \eqd X_2 \\
&\tilde X_1 + \tilde X_2 \eqd X_1 + X_2.
\end{split}
\ee

Let $f(x_1,x_2)$ denote the joint probability density function of
$(X_1,X_2)$. It is supported on the set $\Delta \in [0,1]^2$ such that
\be
\label{domain Delta} 
x_1 \geq x_2 \geq 1-x_1-x_2 \geq 0, 
\ee 
see Fig.\ \ref{fig holroyd}. 
We can find a square in the set $\Delta$, and define
the function $g(x_1,x_2)$ that takes values $\{-1,0,+1\}$ as shown in
Fig.\ \ref{fig holroyd}.  If $f$ is positive on $\Delta$, then $\tilde f =
f + \eta g$ is positive for $\eta$ small enough. The function $\tilde
f$ is the probability density function for $(\tilde X_1, \tilde X_2)$.

The marginals for $X_1, \tilde X_1$ are obtained by integrating $f,
\tilde f$ along vertical lines. They are clearly identical. Same for
the marginals for $X_2,\tilde X_2$, obtained by integrating along
horizontal lines. And same for the marginals for $X_1+X_2, \tilde X_1
+ \tilde X_2$, obtained by integrating along oblique lines of slope
$-1$.

\bfig
\includegraphics[scale=.5]{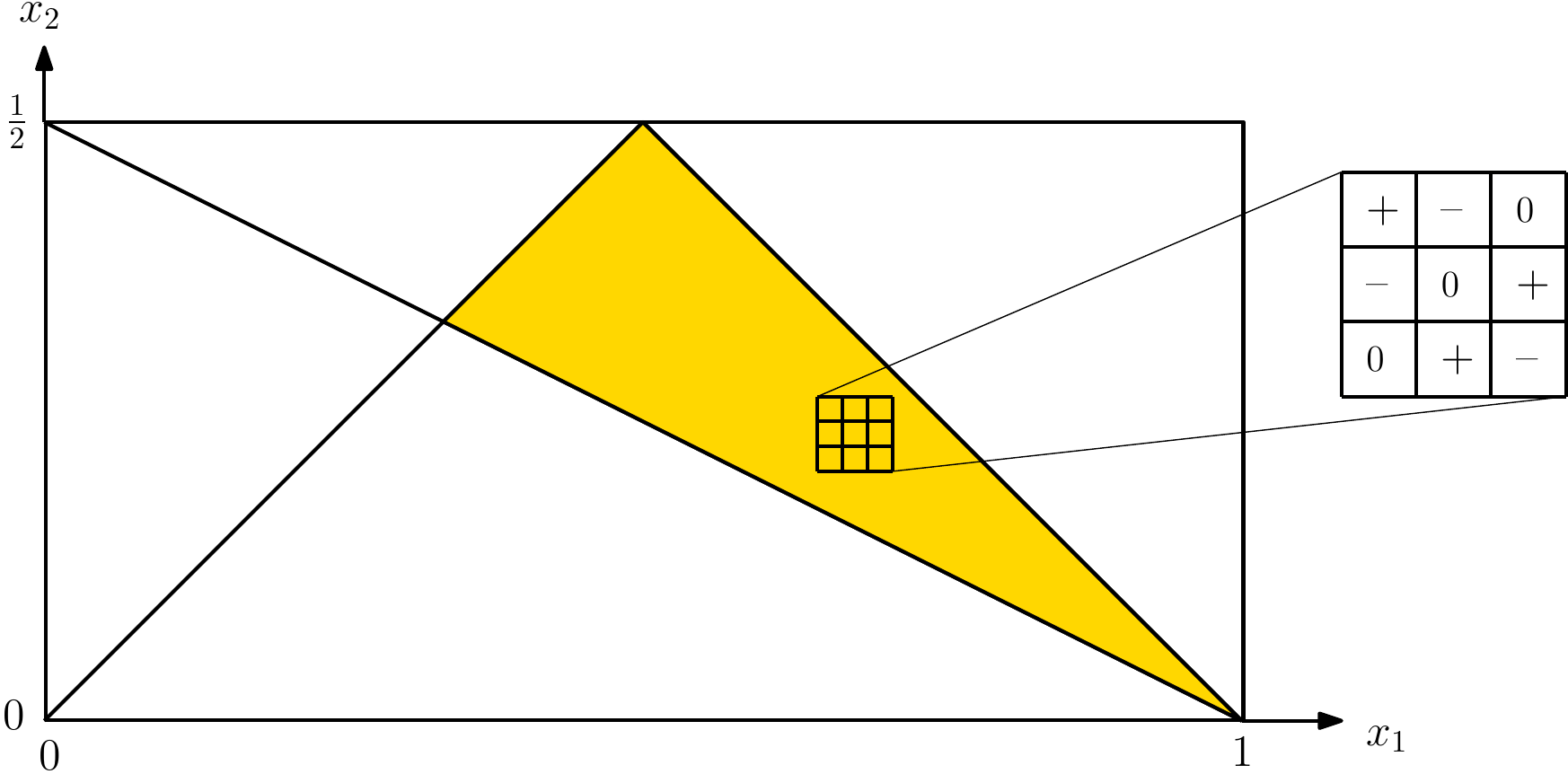}
\caption{The domain $\Delta$ characterised by \eqref{domain Delta}, 
and the square that defines the function $g$.}
\label{fig holroyd}
\efig

Holroyd's example can be generalised to random partitions with
infinitely many elements as follows. Let $a \in (\frac34,1]$. Choose
$(X_1,X_2,X_3)$ with the constraint $X_1 + X_2 + X_3 = a$; then choose
an arbitrary random partition on the remaining interval $[0,1-a]$. The
domain \eqref{domain Delta} is replaced by $x_1 \geq x_2 \geq
a-x_1-x_2 \geq 1-a$ (it is nonempty for $a>\nicefrac34$). The same
argument then applies. It is also possible to take $a$ to be random.

\section{Acknowledgements}

We thank Christina Goldschmidt for discussions and for pointing out the examples from the Brownian bridge and pseudo-bridge, and the reference \cite{Pit3}; Jon Warren for further discussions about the pseudo-bridge; Jeff Steif and
Johan Tykesson for discussions about their paper \cite{ST};
and Alexander Holroyd for explaining his example.
We also thank the three referees for positive and helpful comments.
We gratefully acknowlege support from the UoC Forum \emph{Classical and quantum dynamics of interacting particle systems}.
JEB gratefully acknowledges support from \emph{Vetenskapsr{\aa}det} grant
2015-0519 and \emph{Ruth och Nils-Erik Stenb\"acks stiftelse}. CM is grateful to the EPSRC for support through the fellowship EP/R022186/1.

\renewcommand{\refname}{\small References}
\bibliographystyle{symposium}

\end{document}